\documentclass{article}
\usepackage{latexsym,amsfonts,euscript,amsmath,graphicx}
\usepackage{amssymb}
\usepackage{amsthm}
\usepackage{stmaryrd}
\usepackage{mathrsfs}
\usepackage{leftidx}
\usepackage{indentfirst}
\usepackage{color}

\usepackage[colorlinks,
linkcolor=blue,
anchorcolor=blue,
citecolor=blue]{hyperref} 

\DeclareFontFamily{U}{mathx}{\hyphenchar\font45}
\DeclareFontShape{U}{mathx}{m}{n}{
      <5> <6> <7> <8> <9> <10> gen * mathx
      <10.95> mathx10 <12> <14.4> <17.28> <20.74> <24.88> mathx12
      }{}
\DeclareSymbolFont{mathx}{U}{mathx}{m}{n}

\DeclareMathSymbol{\bigplus}{1}{mathx}{"A0}
\DeclareMathSymbol{\bigtimes}{1}{mathx}{"A1}

\usepackage{fonttable}

\makeatletter 
\@addtoreset{equation}{section}
\makeatother  

\def\la{\langle}\def\ra{\rangle}

\newcommand{\rank}{{\operatorname{rank}}}
\newcommand{\class}{{\operatorname{c}}}

\newcommand{\Aut}{{\operatorname{Aut}}}

\def\irr#1{{\rm Irr}(#1)}
\def\cent#1#2{{\bf C}_{#1}(#2)}
\def\syl#1#2{{\rm Syl}_#1(#2)}
\def\norm#1#2{{\bf N}_{#1}(#2)}
\def\oh#1#2{{\bf O}_{#1}(#2)}

\def\fitt#1{{\bf F}(#1)}
\def\z#1{{\bf Z}(#1)}

\def\nor{\trianglelefteq}

\def\V#1{{\bf{V}}(#1)}
\def\N#1{{\bf{N}}(#1)}
\def\P#1{{\rm P}_{\bf{ v}}(#1)}
\def\o#1{\overline{#1}}

\newtheorem*{thmA}{\bf Theorem A}

\newtheorem{lem}{ \bf Lemma}[section]

\newtheorem{pro}[lem]{\bf Proposition}
\newtheorem{thm}[lem]{\bf Theorem}
\newtheorem{rem}[lem]{\bf Remark}

\newtheorem{cor}[lem]{\bf Corollary}
\newtheorem{hy}[lem]{\bf Setting}

\title{Finite groups with a small proportion of vanishing elements
\thanks{{\bf Acknowledgement:} The first and second authors gratefully acknowledge the support of China Scholarship Council (CSC) and the first author was supported by the Natural Science Foundation for the Universities in Jiangsu Province (No. 23KJB110002). The third author was partially supported by INDAM-GNSAGA.}
}

\author{ 
   Dongfang Yang\footnote{Department of  Mathematics, Changshu Institute of Technology, Changshu, China. email: dfyang1228@163.com},
   Yu Zeng\footnote{Department of  Mathematics, Changshu Institute of Technology, Changshu, China. email: yuzeng2004@163.com}
  and Silvio Dolfi\footnote{Dipartimento di Matematica e Informatica, Universit\`a di Firenze, Italy. email: silvio.dolfi@unifi.it}}

\date{}

\begin{document}
\maketitle

\vskip 1cm

\begin{abstract}
  The function $\P G$, measuring the proportion of the elements of a finite group $G$ that are zeros
  of irreducible characters of $G$, takes (as proved in ~\cite{ZYD2022}) only values $\frac{m-1}{m}$, for $1 \leq m \leq 6$,
  in the interval $[0, \P {A_7})$.
  In this paper, we give a complete classification of the finite groups $G$ such that $\P G=\frac{m-1}{m}$ for $m=1,2,\cdots ,6$.
\end{abstract}

\vskip 5cm

\bigskip


\textbf{2020 MR Subject Classification}\,\, 20C15. 
 \pagebreak

\section{Introduction}
We say that an element $g$ of a finite group $G$ is a \emph{vanishing} element of $G$ if there exists an irreducible character $\chi$ of $G$ such that
$\chi(g) = 0$, and we denote by $\V G$ the set of the vanishing elements of $G$.

A natural way to measure the relative abundance of vanishing elements in a finite group $G$ is to consider the proportion
$$\P G = \frac{|\V G|}{|G|}$$
which can  be seen as the probability  of getting a vanishing element by picking at random (with uniform distribution) an element in $G$.
Relevant questions in this context are those concerning  the possible values of the function $\P{G}$,
and the conditions on the structure of a group $G$ that are related to specific values of $\P G$.
We recall that by Burnside's classical theorem on zeros of irreducible characters, $G$ is abelian if and only if $\P G = 0$.
Strengthening this result, in~\cite{MTV} L. Morotti and H. Tong-Viet  show that a finite group $G$ is abelian if and only if $\P G < \frac{1}{2}$, and
that $\P G = \frac{1}{2}$ if and only if  $G$ is an $\mathcal{A}$-group (i.e. a solvable group whose Sylow subgroups are all abelian) and  $G/\z G$ is a Frobenius group with complement of order $2$.
They also prove (\cite[Theorem 1.6]{MTV}) that if $\P G \leq \frac{2}{3}$, then
$\P G \in \{0,  \frac{1}{2}, \frac{2}{3} \}$ and conjecture that,  for $G$ in the class of finite groups,  the only values of  $\P G$ smaller than $\P{A_7}$, where $A_7$ is the alternating group on $7$ letters,  are
of the form $\frac{m-1}{m}$ for some integer $1 \leq m \leq 6$.
This conjecture has been confirmed in~\cite{ZYD2022}, after that  A. Moret\'o and P.H. Tiep proved in~\cite{MT}  that
a group $G$ is necessarily a solvable if $\P G < \P{A_7} =  \frac{1067}{1260} \simeq 0.846$. 

In this paper, we give a complete classification of the  finite groups $G$ such that $\P G <\P{A_7}$.
Before stating  the main result of this paper, we need to introduce some notation.  
As usual, for an abelian group $A$ and a group $H$, we say that $A$ is an $H$-module if it is given a group homomorphism from
$H$ to $\rm{Aut}(A)$.
 We say that $A$ is a \emph{cyclic} $H$-module if $A$ is generated, as $H$-module, by a single element of $A$.

For a group $H \cong C_6$ (where $C_n$ is  the cyclic group of order $n$) and an abelian $2$-group $A$, we say that
\begin{description}
\item[(T1)] $A$ is \emph{of type (T1)} if $A$ is a cyclic $H$-module of rank $2$, $A \cong C_8 \times C_8$ and, for suitable $x, y \in H$ such that  $o(x) = 3$ and $o(y) = 2$, $\cent Ax = 1$ and $aa^y = (a^4)^x$ for all $a \in A$. 
  
\item[(T2)] $A$ is \emph{of type (T2)} if $A$ is a cyclic $H$-module of rank $4$, $A = A_0A_1$, with $A_0  \cong C_{2^n} \times C_{2^n}$,
  $2^n \geq 8$, $A_1 = [A_0, y] \cong C_4 \times C_4$ and, for suitable $x, y \in H$ such that  $o(x) = 3$ and $o(y) = 2$,
  $\cent Ax = 1$ and $[a^2, y]^x  = a^{2^{n-1}}$ for all $a \in A$.
\end{description}

We say that an abelian $2$-group $A$ is a \emph{homogeneous $H$-module of type (T1) (resp. (T2))}, where $H \cong C_6$,  if $A$ is a direct sum of isomorphic
  $H$-modules of type (T1) (resp. (T2)).

\begin{thmA}
  Let $G$ be a finite group, $Q\in\syl{2}{G}$ and $P\in\syl{3}{G}$.
  Then $\P G<\P{A_7}$ if and only if 
  \begin{description}
  \item[(a)] $G$ is an $\mathcal{A}$-group such that  $[G:\fitt{G}]=m\leq 6$,  and  
   $|\fitt{G}/\z{G}|$ is not divisible by $6$ if $m = 5$.
 \item[(b)]  $Q$ is nonabelian and $G$ has an abelian normal subgroup $A$ such that  
    $Q\cap A  = Z_0 \times D$,
    where $Z_0 = \cent{Q \cap A}P  \leq \z{G}$  and $D = [Q \cap A, P]$, 
  and  one of the following holds:
  \begin{description}
  \item[(b1)] $|G/A|=4$ and $Q \cap A=\z{Q}$. 
 \item[(b2)] $G/A\cong S_3$,  
   and, for some  $x\in P- A$,   $1 \neq D =Z\times Z^{x}$, where $Z=\cent{D}{Q}$ is either elementary abelian or 
   the direct product of a cyclic group of order $4$ and a possibly trivial elementary abelian group.
 \item[(b3)]  $G/A\cong C_6$, 
   $1 \neq D = B\times C$ with $B$ and $C$ are normal subgroups of  $G$ such that  
   $[C,Q,Q]=1$ and if  $B \neq 1$, then  $\exp (B)>\exp (C)$ and either every $y \in Q -A$ acts as the inversion on $C$ and 
   $B$ is a homogeneous $G/A$-module of type {\rm (T1)},  or $B$ is a homogeneous $G/A$-module of type {\rm (T2)}. 
\end{description}
\end{description}
\end{thmA}

In Theorem~\ref{Asmall} and Corollary~\ref{corAgr}, for completeness,  we describe in full detail the groups that satisfy condition
(a) of Theorem~A.

The paper is organized as follows: in Section 2 and Section 3 we collect several preliminary results, while in Section 4
we prove Theorem~A and Theorem~\ref{Asmall}.

Our notation is standard and for character theory   follows~\cite{isaacs1976}.
All groups considered in the paper will be assumed to be finite groups.

\section{Preliminaries}
Let $G$ be a finite group. 
As already mentioned, we write
$$\V G = \{ g \in G \mid \chi(g) = 0 \text{ \ for  \ some \ } \chi \in \irr G\},$$ 
and we denote by $\N G = G - \V G$
its complement in $G$, i.e. the set of the \emph{non-vanishing} elements of $G$. 

We remark that, given a normal subgroup $N$ of a group $G$ and $g \in G$, if $gN \in \V{G/N}$ then
$gN \subseteq \V G$. This implies that $\P{G/N} \leq \P G$, a fact that will be used freely in the rest of the paper. 

We collect some useful results about vanishing elements. 

\begin{lem}[\mbox{\cite[Main Theorem]{lam2000}}]\label{vs}
  Let $m, n$ be a positive integers, $ \alpha_1, \alpha_2, \ldots, \alpha_n$
   $m$-th roots of unity
  and let $p_1, p_2, \ldots,  p_r$ be the (distinct) prime divisors of $m$.
  Then $\alpha_1+\cdots+\alpha_n=0$  only if $n=n_1p_1+\cdots +n_rp_r$ for suitable nonnegative integers $n_i$.
\end{lem}

\begin{lem}[\mbox{\cite[Lemma 2.4]{ZYD2022}}]\label{nonvancent}
  Let $G$ be a group and  $z\in\z{G}$. Then  $x\in \V G$ if and only if  $zx\in \V G$.
 \end{lem}

\begin{lem}[\mbox{\cite[Corollary 1.3]{gruninger}}]\label{nonvanishingsmallgroup}
  Let $g \in H \leq G$ be such that $G =\cent{G}{g}H$. Then $g\in \V G$ if and only if $g\in \V H$.
\end{lem}

\begin{lem}[\mbox{\cite[Lemma 2.5]{ZYD2022}}]\label{dpnv}
  Let  $A$ and $B$  be normal subgroups of the group $G$ such that $A\cap B=1$.
Then, for $a \in A$, we have $a \in \V{G}$ if and only if $aB \in \V{G/B}$.
\end{lem}

\begin{lem}[\mbox{\cite[Theorem A]{isaacs1999}}]\label{vP}
  If $P$ is a $p$-group, $p$ a prime number, then $\N P = \z P$ and hence
  $\P P = \frac{m-1}{m}$, where $m = [P: \z P]$. 
\end{lem}

\begin{lem}[\cite{brough2016}]\label{brough}
 Let $G$ be a finite group and let $P$ be a Sylow $p$-subgroup of $G$, $p$ a prime number.
Then $\z P \cap \oh pG \subseteq \N G$.
\end{lem}

For short, in the rest of the paper we write $\mathfrak{a} = \P{A_7}$.  
Although  $\N G$ is  in general quite far from being a subgroup of the group $G$, it is indeed an abelian normal subgroup under the assumption that $\P G< \mathfrak{a}$, as shown in the following result (which is the key theorem of~\cite{ZYD2022}). 
\begin{thm}[\mbox{\cite[Theorem 4.3]{ZYD2022}}]\label{zyd}
  If $G$ is a finite group such that  $\P G < \mathfrak{a}$, then $\N G$ is an abelian normal subgroup of $G$.
  \end{thm}

\section{Auxiliary results}

Let $A$ be an abelian group and let $\hat{A}=\irr{A}$ be the dual group of $A$.
For  $B \leq A$, let  $B^\perp = \{\alpha\in\hat{A}\mid \alpha(b)=1,~\forall~b\in B\} \leq \hat{A}$, and 
for  $V \leq \hat{A}$, let $V^\perp = \{a\in A\mid \lambda(a)=1,~\forall~\lambda \in V\} \leq A$.
Then $B^\perp \cong \widehat{A/B}$
and $B^{\perp\perp} = B$ (see for instance \cite[V.6.4]{huppertI}). 
Finally, we  recall that for every  $x \in \Aut(A)$,  $x$ acts  on $\hat{A}$ via 
$\alpha^x(a) = \alpha(a^{x^{-1}})$, where $\alpha \in \hat{A}$ and $a \in A$.

\begin{lem}[\mbox{\cite[Lemma 3.1]{ZYD2022}}]\label{pgroup}
 Let $A$ be an abelian group, $y \in \Aut(A)$ such that $o(y) = 2$,  and  let $Q = A \rtimes \la y \ra$.
  Then  
\begin{description}
  \item[(1)] The map $\gamma:A\rightarrow A$, defined by  $\gamma(a)=[a,y]$
  for  $a\in A$, is a group homomorphism with $\mathrm{Im}(\gamma)=[A,y] = Q'$ and $\ker(\gamma)=  \cent Ay =  \z{Q}$.
  Moreover, $A/\cent Ay$ and $[A,y]$ are isomorphic $\la y \ra$-modules, on which $y$ acts as the inversion. 
  \item[(2)] $[A,y]^{\perp}= \cent{\hat{A}}y$ and $[\hat{A},y]^{\perp}=\cent Ay$.
  \item[(3)] The groups $[\hat{A},y]$ and $ [A,y]$ are isomorphic. 
  \end{description}
\end{lem}

In the following, we denote by $\class(Q)$ the nilpotency class of a nilpotent group $Q$ and by ${\bf Z}_2(Q)$ the second term of the upper central series of $Q$. 
If $A$ is an abelian $p$-group, $p$ a prime, we denote by $\exp(A)$ the exponent of $A$ (the largest order of an element of $A$, in this case) and we define, for $i \geq 0$, 
$\Omega_i(A) = \{ a \in A \mid o(a) \text{ \ divides \ } p^i \}$, which is a characteristic subgroup of $A$.  
We now state a couple of lemmas for later use.

\begin{lem}[\mbox{\cite[Lemma 2.6]{ZYD2022}}]\label{ind}
  Let $A$ be a normal subgroup of the group $G$.
Then, for $a \in A$, $a \in \V G$ if and only if there exists a character $\alpha \in \irr A$ such that  $\alpha^G(a) = 0$.   
\end{lem}

\begin{lem}[\mbox{\cite[Lemma 3.5]{ZYD2022}}]\label{5/6}
  Let $A$ be an abelian normal $2$-subgroup of the group $G$ such that $[G:A] =6$, and let $Q\in\syl{2}{G}$ and $P\in\syl{3}{G}$.
  Then 
  \begin{description}
  \item[(1)]  There exists an element $y\in \norm{Q}{P}$ such that $y \not\in A$ and $y^2\in \cent{A}{P}$.
  \item[(2)] If $\cent AP = 1$ and either   $\exp(Q') \leq 2$   or $\class(Q) \leq 2$, then $A \subseteq \N G$. 
  \end{description}
Assume further that $P$ acts nontrivially on $A$ and that $Q$ is nonabelian. Then
  \begin{description}
  \item[(3)]  $G- A\subseteq \V G$. 
  \item[(4)]  If $\P G \leq \mathfrak{a}$, then $\cent AP \leq \z G$. 
  \end{description}
\end{lem}

In the rest of the section, we will address the most complicated cases arising in the process of proving Theorem~A.
In order to avoid repetitions, we introduce the following setting.

\begin{hy}\label{hyp}
  Let $A$ be an abelian normal $2$-subgroup of the group $G$ such that $[G:A] = 6$ and assume that  $A$ has a complement $H$ in $G$.
  Let $P$ be the Sylow $3$-subgroup of $H$, $y$ an involution of $H$ and  $Q=A\la y\ra$ a Sylow $2$-subgroup of $G$.
  Assume that  $Q$ is nonabelian and that   $\cent AP =1$.
\end{hy}

We remark that,  by part (3) of Lemma \ref{5/6},   if  $G$ satisfyies Setting \ref{hyp}, then 
$\P G=\frac{5}{6}$ if and only if  $A \subseteq \N G$.  

Given an abelian $p$-group $A$, we denote by $\rank(A)$ its \emph{rank} (i.e. the number of factors in a decomposition of $A$ as a direct product of cyclic groups) and, for $a \in A$, $k \in \mathbb{Z}$ and
$x \in \Aut(A)$, we will write, for short,  $a^{kx}$ instead of $(a^k)^x$; in particular, we will write $a^{-x}$ for $(a^x)^{-1}$.

\begin{rem} \label{3.3a}
\normalfont If a cyclic group $P = \la x \ra$ of order $3$ acts on an abelian $2$-group $A$ and $\cent AP =1$, then     
   $aa^xa^{x^2} = 1$  for every $a \in A$,  because $aa^xa^{x^2} \in \cent AP$.
Considering $A$ as a $P$-module, 
if $B\neq 1$ is a $P$-submodule of $A$ and
$\rank(B) \leq 2$, then   $B$ is an indecomposable $P$-module and hence $B$ is homocyclic of rank $2$.
Thus $A$, being a direct sum of indecomposable $P$-modules, has even rank.
Moreover, 
$B$ is a uniserial $P$-module whose only submodules are the subgroups $\Omega_i(B)$, $i \geq 0$,  and whose  composition factors are all isomorphic to $C_2 \times C_2$ (see for instance \cite{harris}).
As a consequence, for every $ a \in A$, $a \neq 1$,  the $P$-submodule of $A$ generated by $a$ is   $\la a, a^x\ra = \la a \ra \times \la a^x \ra$.
\end{rem}

If $A$ is an $H$-module, we say that $B$ is a \emph{cyclic} submodule of $A$ if there exists a single element $b \in A$ such that
$B = \la b \ra^H$, i.e. $B$ is generated  by $b$ as an $H$-module. Observe that this is equivalent to $B = \la b \ra^G$,
the normal closure of $\la b \ra$ in the semidirect product $G = A \rtimes H$.

\begin{rem}\label{3.3b}
 \normalfont  In the situation of Setting~\ref{hyp}, 
a nontrivial cyclic $H$-submodule of $A$ has rank either $2$ or $4$.
  In fact, writing $H = \la x, y \ra$, where $P = \la x \ra$, the condition $\cent AP = 1$ implies that $bb^xb^{x^2} = 1$ for every
  $b \in A$. So, for $1 \neq b \in A$, $B = \la b\ra^H = \la b, b^x, b^y, b^{yx} \ra$ has rank at most $4$ and hence
  $\rank(B) \in \{ 2,4 \}$ by the previous remark. We also note that $\rank(B) = 4$ if and only if $b^y \not\in \la b, b^x\ra$,
  and that $\rank(B) = 2$ if and only if $B = \la b \ra \times \la b^x\ra$.  
\end{rem}

We denote, for short,  by $(C_n)^k$ the direct product of $k \geq 0$ copies of the cyclic group $C_n$. 

\begin{lem}[\mbox{\cite[Lemma 3.7]{ZYD2022}}]\label{XYtypeelementary}
  Assume Setting~\ref{hyp}  and that $A \subseteq \N G$. Let $Z = \z Q$ and $P = \la x \ra$.
  Then
  \begin{description}
  \item[(1)]  If $H \cong C_6$, then $Z \nor G$ and $\exp(A/Z) = 2^{c-1}$, where  $c = \class(Q) \leq 3$.  
  \item[(2)]  If $H \cong C_6$ and $B$ is  a $P$-submodule of $A$ such that  $\rank(B) = 2$ and  $A_0 = B \times B^y$
    has index at most $4$ in $A$, then $\exp(B)=2$.
  \item[(3)] If  $H\cong S_3$,   $A=Z\times Z^x$ and $\rank(A) \leq 4$, then $\exp(A) \leq 4$ and  $Z$
  is not isomorphic to $(C_4)^2$.  
  \end{description}   
\end{lem}

\subsection{$S_3$ case}

In this subsection, we classify the groups $G$ that satisfy  Setting \ref{hyp} such that  $H\cong S_3$ and $\P G=\frac{5}{6}$.

\begin{pro}\label{S3classify}
  Assume Setting \ref{hyp} and that $H$ is isomorphic to $S_3$.
  Let $P=\langle x\rangle$ and $Z = \cent{A}{y}$.
  Then $A= Z \times Z^x$. 
  Moreover, $A \subseteq \N G$ if and only if  $Z$ is isomorphic to a nontrivial subgroup of  $C_4\times (C_2)^t$,  for some  integer $t \geq 0$.
\end{pro}
\begin{proof}
 We start by observing that $Z = \cent{A}{y}=\z Q$ is nontrivial.
  As $H= \la y,y^x\ra$, we have  
   $$Z\cap Z^x=\cent{A}{y}\cap\cent{A}{y}^x=\cent{A}{\la y,y^x\ra}=\cent{A}{H}\leq \cent{A}{P}=1.$$
   An application of part (1) of Lemma~\ref{pgroup} to the action of $y$ on $\hat{A}$ yields 
  $\mu^y=\mu^{-1}$ and hence   $\mu^{xy}=\mu^{-x^2}\neq \mu^{-x}$   for every $\mu\in [\hat{A},y]- \{ 1_A \}$, 
   so $[\hat{A},y]\cap [\hat{A},y]^x=1_A$.
    Since by Lemma~\ref{pgroup} $[\hat{A},y]^\perp= Z$,
    it follows that
  \begin{center}
    $A=(1_A)^\perp=([\hat{A},y]\cap [\hat{A},y]^x)^\perp=Z Z^x=Z\times Z^x$.
  \end{center}
  Write $Z=C\times D$, with $C=\la c\ra$ such that $o(c) = \exp(Z)$. 
  Then $A=V\times W$ where $V=C\times C^x $ and $W=D\times D^x$. As $z^{xy} = z^{yx^2}= z^{x^2} =(zz^x)^{-1}$ for every $z\in Z$, we have both  $V \nor G$ and $W \nor G$.

\smallskip

  Let us suppose, first, that $A \subseteq \N G$.
  Take $d \in D$, and let $Z_0=\langle c\rangle\times \langle d\rangle$. 
  Since $o(c)=\exp(Z)$, 
  in order to conclude that $Z$ is isomorphic to a subgroup of $C_4\times (C_2)^t$ for some $t\geq 0$,
  it suffices to show that $Z_0$ is isomorphic to a subgroup of  $C_4 \times C_2$.
  Let $A_0=Z_0\times Z_0^x$, $Q_0=A_0 \rtimes \langle y\rangle$ and $G_0=A_0  \rtimes H$.
  As  $G=\cent{G}{a_0} G_0$ for  every $a_0 \in A_0$,  
  by Lemma \ref{nonvanishingsmallgroup}
  we deduce that  $A_0 \subseteq \N {G_0}$.
  Note that $(c^{x})^{y}= c^{yx^2}= c^{x^2} \neq c^x$,
  so $Q_0$ is nonabelian and hence  $G_0$ satisfies Setting \ref{hyp}.
  Now, an application of part (3) of Lemma \ref{XYtypeelementary} to $G_0$ yields that $Z_0$ is isomorphic to a subgroup of  $C_4 \times C_2$. 
    \smallskip

    Conversely, we assume  that $Z$ is isomorphic to a subgroup of $C_4 \times (C_2)^t$ for some $t \geq 0$, and 
    we show that $A \subseteq \N G$.
    Recall that $A=V\times W$,  where $V=\la c\ra\times \la c^x\ra$ with $o(c)=\exp(Z)$.
    So, $o(c)\le 4$ and $\exp(W)\leq 2$.
    If $o(c)\le 2$, then $\exp(A)\le 2$, and hence part (2) of Lemma~\ref{5/6} yields $A \subseteq \N G$.
    So, we may assume that $o(c)= 4$. 
    Let $a\in A$ and $\alpha\in\irr{A}$, and write $a=c^kc^{lx}w$, where $w\in W$, $k,l \in \mathbb{Z}$.
    By Lemma \ref{ind}, it suffices to prove that  $\alpha^G(a)\neq 0$.
    Assume the contrary, i.e.
\begin{equation}\label{eqs_3}
    \alpha^G(a)=  \alpha(a)+\alpha(a^x)+\alpha(a^{x^{2}})+\alpha(a^y)+\alpha(a^{yx})+\alpha(a^{yx^{2}})=0.  
  \end{equation}
where  $\alpha(a)\alpha(a^{x})\alpha(a^{x^{2}})=\alpha(aa^{x}a^{x^{2}})=1$
  and $\alpha(a^{y})\alpha(a^{yx})\alpha(a^{yx^{2}})=\alpha(a^{y}a^{yx}a^{yx^{2}})=1$. 
    As $\alpha$ is a linear character of $A$ and  $\exp(A)\le 4$, the values of $\alpha$ (and of any of its conjugate characters) are $4$-th roots of unity.
    Thus, an application of part (1) of Lemma 2.3 of~\cite{ZYD2022} to (\ref{eqs_3}) yields that 
    $\max\{o(\alpha^{x^i}([a,y])) \mid 0\le i\le 2 \}=4$ and hence,
    as $o(\alpha^{x^i}([a,y]))$ divides $ o([a,y])$, that $o([a,y])=4$.
    Since at least one of the elements $\alpha([a, y]^{x^i}) = \alpha(a^{-x^i})\alpha(a^{yx^i})$ has order $4$ (in $\mathbb{C}^{\times}$), for $0 \leq i \leq 2$, it follows that  $\max\{o(\alpha(a^{h}))\mid h\in H \}= 4$. 
    Since $\alpha^G = (\alpha^h)^G$, for $h \in H$, up to substituting $\alpha$ with a suitable conjugate $\alpha^h$,
    we may assume that both  $\alpha(a)$ and $\alpha(a^x)$ have order $4$, 
   so $\alpha(a^{2})=\alpha(a^{2x})=-1$.
   By~(\ref{eqs_3}),    part (2) of Lemma 2.3 of~\cite{ZYD2022} implies that
$\alpha(a^{yx^i})^2  = 1$ for $0 \leq i \leq 2$. In particular, $\alpha(a^{2y}) = \alpha(a^y)^2 = 1$.    
    Recalling that  $a=c^{k}c^{lx}w$,  as $[a,y]=[c^{k}c^{lx}w,y]=[c^{lx},y] [w, y] = c^{-lx}c^{lx^2}[w, y]$ has order $4$ and 
     $o([w,y])\le 2$ (as $[w,y]\in W$), we deduce that $l$ is odd. 
     As $o(c) = 4$, then $(c^l)^2= c^2$ and hence
     $a^2=c^{2k}c^{2x}$.
     If $k$ is even, then $c^{2k}=1$ and $a^2=c^{2x}$;
     so, recalling that $c^y=c$ and that $xy=yx^2$,
     we have $1=\alpha(a^{2y})=\alpha(c^{2xy})=\alpha(c^{2x^2})=\alpha(a^{2x})=-1$, a contradiction.
     If $k$ is odd, then $a^2=c^2c^{2x}$ and $1=\alpha(a^{2yx^2})=\alpha(a^{2xy})=\alpha(c^{2xy}c^{2x^2y})=\alpha(c^{2x^2}c^{2x})=\alpha(a^{2x})=-1$, again a contradiction. 
       Thus, $\alpha^G(a)\neq 0$, as desired.
\end{proof}

\subsection{$C_6$ case}

\begin{hy}\label{hyp1}
  Assume  Setting \ref{hyp} with $H$  isomorphic to $C_6$ and  $c(Q) \geq 3$.
\end{hy}

In this subsection, we are going to classify the groups $G$ satisfying  
Setting \ref{hyp1} and  $\P G=\frac{5}{6}$.

\begin{rem}\label{3.6}
\normalfont  If $G$ satisfies Setting~\ref{hyp1} and $A \subseteq \N G$, then  part (1) of Lemma \ref{XYtypeelementary} yields
  that $c(Q)=3$ and that $\exp(A/\z{Q}) = 4$.  
\end{rem}

First,  we deal  with the case of  Setting \ref{hyp1} where $A$ is a cyclic $H$-module: see Proposition~\ref{classcyclicmodule}.
In the following, we denote by $\zeta_n$ the primitive $n$-th root of unity $e^{\frac{2\pi}{n}i} \in \mathbb{C}^{\times}$.  
We start with  two lemmas.

\begin{lem}\label{Xtypeclassification}
  Assume Setting~\ref{hyp1}.
  Suppose that $\rank(A) = 2$ and that  $\exp(A) = 8$.
  Then $A \subseteq \N G$ if and only if 
  for a suitable generator $x$ of $P$, $aa^y = a^{4x}$ for all $a\in A$.
  \end{lem}
\begin{proof}
  We observe that  $Q \nor G$ and then both $Q'$ and $Z = \z Q$  are $P$-submodules of
  the uniserial $P$-module $A \cong C_8 \times C_8$ (see Remark \ref{3.3a}). 
  Since $\class(Q) > 2$, $Q'$ is not contained in $Z$ and 
  hence $Z = \Omega_1(A)$ and $Q' = A^2$.
  Since $\class(A^2\la y \ra) = 2$,  part (2) of Lemma~\ref{5/6} applied to the subgroup $A^2H$ yields  $A^2 \subseteq \N{A^2H}$
  and hence $A^2 \subseteq \N G$ by Lemma~\ref{nonvanishingsmallgroup}.
Let  $a_0 \in A$ be such that $o(a_0) = 8$; by the Frobenius action of $P = \la x \ra$ on $A$, we have   $A=\la a_0\ra \times \la a_0^x\ra$. 
Note that  $a_0a_0^y \in  \cent Ay = Z = \{1,  a_0^4,a_0^{4x}, a_0^{4x^{2}}\}$.

Assume first that $A \subseteq \N G$.
We claim that $a_0a_0^y \in \{a_0^{4x}, a_0^{4x^{2}}\}$.
Otherwise, $a_0a_0^y \in \{1, a_0^4\}$, so  $a_0^y=a_0^k$ for a suitable $k \in \{-1, 3\}$.
Let $\alpha\in\irr{A}$ be such that $\alpha(a_0)=\zeta_8$ and $\alpha(a_0^x)=\zeta_8^5$;
so, $\alpha(a_0^{x^2}) = (\alpha(a_0)\alpha(a_0^x))^{-1} =\zeta_8^2$.
  Hence, 
  \[
  \alpha^G(a_0)=\sum\limits_{0\leq i\leq 2}(\alpha^{x^i}(a_0)+\alpha^{x^iy}(a_0))=\sum\limits_{0\leq i\leq 2}(\alpha^{x^i}(a_0)+\alpha^{x^i}(a_0^k))=\sum\limits_{j \in \{1,2,3,5,6,7\} }\zeta_8^j=0,   
  \]
  so $a_0\in \V G$ by Lemma~\ref{ind}, a contradiction.
  Hence, for a suitable generator $x$ of $P$, $a_0a_0^{y}=a_0^{4x}$.
  We now show that the action of $x$ on $A$ is as claimed. 
  In fact, for every $a \in A$, $a=a_0^i a_0^{jx}$ for some integers $i$ and $j$.
  As $xy=yx$,
  $$aa^{y}=a_0^i a_0^{jx}(a_0^i a_0^{jx})^{y}=(a_0a_0^{y})^i(a_0a_0^{y})^{jx}=a_0^{4ix}a_0^{4jx^{2}}=a^{4x}.$$

  \medskip
  Assume conversely that for a suitable generator $x$ of $P$, $aa^{y}=a^{4x}$ for every $a \in G$.
  In particular, $aa^{y}=a^{4x}$ for every $a \in A-A^{2}$.
  Let $a \in A-A^{2}$.
  As $A^{2} \subseteq \N G$, it suffices to show that $a \in \N G$.
  Arguing by contradiction,   by Lemma~\ref{ind} there exists a character $\alpha \in \irr A$ such that
  $\alpha^G(a) = 0$.
  Let $K = \ker (\alpha)$ and observe that $Z \not\leq K$: otherwise, writing $\overline{G} = G/Z$,
  we have $\alpha \in \irr{\overline A}$ and $\alpha^{\o G}(\o a) =0$, 
  so $\overline{a} \in \V{\overline{G}} \cap \overline{A}$, against part (2) of Lemma~\ref{5/6} since $\class(\overline{Q}) = 2$.  
  Hence, as $a^4a^{4x}a^{4x^2} = 1$ and $\{\alpha(a^{4x^i}) \mid 0 \leq i \leq 2 \} = \{1, -1\}$, up to possibly exchanging $\alpha$ with some conjugate character $\alpha^{x^i}$
(observing that $\alpha^G = (\alpha^{x^i})^G$ for every $i$), 
 we can assume $\alpha(a^{4x}) =1$ and $\alpha(a^4)=\alpha(a^{4x^2}) = -1$.
  Recalling that $a^y=a^{4x}a^{-1}$, we have that $\alpha^{x^i}(a^y)=\alpha^{x^i}(a^{4x})\o {\alpha^{x^i}(a)} = \alpha(a^{4x^{1-i}})\o{\alpha^{x^i}(a)}$ for $0\leq i\leq 2$ and then 
    \[
      \begin{split}
        \alpha^G(a) &=\alpha(a) + \alpha^x(a) + \alpha^{x^2}(a) + \alpha(a^y) + \alpha^x(a^y) + \alpha^{x^2}(a^y) \\ 
        &= \alpha(a) + \alpha^x(a) + \alpha^{x^2}(a) +  \overline{\alpha(a)} -\overline{\alpha^x(a)} - \o {\alpha^{x^2}(a)}
      \end{split}  
    \]
Hence, the assumption $\alpha^G(a) = 0$
gives
$$\alpha(a) +\overline{\alpha(a)} = \overline{\alpha^x(a)} - \alpha^x(a) + \overline{\alpha^{x^2}(a)} - \alpha^{x^2}(a).$$
Observing that the first member of the above equality is real, while the second member is purely imaginary, we deduce
that $\overline{\alpha(a)} = -\alpha(a)$ and we get  $\alpha(a)^4 = \alpha(a^4) = 1$, a contradiction as $\alpha(a^4)= -1$.
Hence, $a \in \N G$ and we conclude that $A \subseteq \N G$. 
\end{proof}

\begin{lem}\label{Ytypevanishelement}
  Assume Setting~\ref{hyp1}. 
Let $A = CD$,  where  $C$ is a  $P$-submodule of rank $2$ of $A$, $D = [C,y]$ and $C \cap D = D^2 = \Omega_1(D)$.  
If $2^n = \exp(A) \geq 8$, then $A \subseteq \N G$ if and only if for a suitable generator $x$ of $P$,
$a^{2^{n-1}} = [a^2,y]^x$ for every $a \in A$.
\end{lem}
\begin{proof}
 See \cite[Lemma 3.9]{ZYD2022}.
\end{proof}
The next result is a key for the classification we are after.

\begin{pro}\label{classcyclicmodule}
  Assume Setting \ref{hyp1} and that $A$ is a cyclic $H$-module, say $A=\la a_0\ra^G$ for some $a_0 \in A$.  
  Then $A \subseteq \N G$ if and only if, for a suitable generator $x$ of $P$,  one of the following holds.
  \begin{description}
    \item[(1)] $\rank(A) = 2$, $A = \la a_0, a_0^x \ra \cong (C_8)^2$ and $aa^y = a^{4x}$ for all $a \in A$; i.e. $A$ is of type {\rm (T1)}. 
    \item[(2)] $\rank(A) = 4$, $A=CD$ where $C=\la a_0, a_0^x\ra\cong (C_{2^{n}})^2$,  $n\geq 3$, $D=[C,y]\cong (C_4)^{2}$ and
      $a^{2^{n-1}} = [a^2,y]^x$ for every $a \in A$; i.e. $A$ is of type {\rm (T2)}. 
    \end{description}
    Moreover, if $A \subseteq \N G$, $[A,y,y]$ is the unique irreducible  $H$-submodule of  $A$.
\end{pro}
\begin{proof}
  If either (1) or (2) holds, then $A \subseteq \N G$ follows directly by Lemma \ref{Xtypeclassification} and Lemma~\ref{Ytypevanishelement}.

  Assume, conversely,  that $A \subseteq \N G$.
  Recall that, by  Remark~\ref{3.6}  $c(Q)=3$ and  $\exp(A/\z{Q}) = 4$.
  As $Q=A \langle y\rangle$ and $A=\langle a_0\rangle^G$, $a_0\neq 1$.
  So, $\rank(A)=2$ or $\rank(A) = 4$ by Remark \ref{3.3b}.
  
  Suppose first that $\rank(A)=2$.
  Then $A$ is a uniserial $P$-module (see Remark~\ref{3.3a}) and, as $Q = A\la y \ra$ has class $3$, $\z{Q} < Q'$.  
  Recall that by part (1) of Lemma \ref{pgroup}, $Q' = [A, y]$ and  $A/\z{Q} \cong (C_4)^2$ are isomorphic $\la y \ra$-modules, on which 
  $y$ acts as the inversion. It follows that  $\z Q = \Omega_1(Q')$ and
  that $\exp(A)=8$. Hence,  by Lemma \ref{Xtypeclassification} we have case (1).
  Moreover, since  $A$  is a uniserial $P$-module of rank $2$,  $[A,y,y]=\Omega_1(A)$ is the unique irreducible $P$-submodule, and hence $H$-submodule,   of $A$.

  \smallskip
  Suppose now that $\rank(A)=4$.
  Let $C=\la a_0,a_0^x\ra$ and $D=[C,y]$; so, by Remark~\ref{3.3b}, $\rank(D) = 2$.
  Observe that $A=CC^y=CD$.
   As $C\cap C^{y}\unlhd G$, we set $\o G=G/(C\cap C^{y})$. 
   Since $A\subseteq\N G$ implies $\o A \subseteq \N {\o G}$,
   applying part (2) of Lemma \ref{XYtypeelementary} to 
   $\o A=\o C \times \o {C}^{\o y}$,
   we have that $\o C\cong (C_2)^2$ and hence $A/C \cong (C_2)^2$.
   Write $\exp(A) = o(a_0) = 2^n$ and let $Z = \z{Q} = \cent Ay$.
   By Lemma \ref{pgroup} and part (1) of Lemma \ref{XYtypeelementary} $D=[A,y]\cong A/Z\cong (C_4)^2$, so $n \geq 2$ and
   $C\cap D=D^2=\Omega_1(D)$.
    We claim that $n \geq 3$; in fact, otherwise $|A| = |C||D/C\cap D| = 2^6$ and from $|A/Z| = 2^4$ it follows
$|Z| = 2^2$, so $Z = \Omega_1(D) = \Omega_1(C)$ and $A/Z \cong C/Z \times D/Z$ is elementary abelian, against $A/Z \cong (C_4)^2$. 
An application of Lemma \ref{Ytypevanishelement} to $A$ yields that for a suitable generator $x$ of $P$, $a^{2^{n-1}}=[a^{2},y]^x$ for all $a \in A$. Hence, we have case (2). 

   Finally, by Lemma~\ref{pgroup} $y$ acts as the inversion on $[A, y] = D \cong (C_4)^2$, so $[A, y,y] = D\cap Z$ is an irreducible $H$-module.
   In order to prove that $D\cap Z$ is the only irreducible $H$-submodule of $A$, it is enough to show that $\rank(Z) = 2$ as any irreducible $H$-submodule of $A$ is contained in $Z$. 
   We observe that if $\rank(Z) > 2$, then $\rank(Z) = \rank(A) = 4$ by Remark \ref{3.3b} and all the involutions of $A$ would be fixed by $y$. 
   Write $w=[a_0,y]^{x}$ and $e=a_0^{2^{n-2}}w$. 
   Observe that $e\neq 1$ as $w\notin C$, and 
  $e^2 = a_0^{2^{n-1}}w^2 = a_0^{2^{n-1}}w^{-2}  = 1$, so $e$ is an involution.
  Moreover, 
   \begin{center}
    $[e,y]=[a_0^{2^{n-2}}w,y]=w^{2^{n-2}x^{-1}}[w,y]=w^{2^{n-2}x^{-1}}w^{-2}\ne 1$,
   \end{center} 
 because $\cent Dx = 1$. Hence, $\rank(Z) = 2$, concluding the proof. 
 \end{proof}

We observe that in part (2) of Proposition \ref{classcyclicmodule}, $n$ can be any integer larger than 2.

\begin{lem}\label{kerfreeA}
  Assume Setting \ref{hyp1}.
  If $A \subseteq \N G$ and there exists a character  $\alpha\in\irr{A}$ such that $\ker(\alpha^G)=1$, then
$A$ is a   cyclic $H$-module. 
\end{lem}
\begin{proof}
  
  Write $K=\ker(\alpha)$. 
  Then $K^\perp \cong \widehat{A/K}\cong A/K$ is a cyclic group of order $o(\alpha)$ and, as $\alpha\in K^\perp$, we see that 
   $K^\perp=\la \alpha\ra$. It follows that, for $h \in H$, $(K^{\perp})^h=(K^h)^\perp = \la \alpha^h \ra$.

   Let $P=\langle x\rangle$. As $\cent Ax = 1$, by Brauer's permutation lemma $\cent{\hat{A}}x = 1_A$ and hence $\lambda\lambda^x\lambda^{x^2}=1_A$ for all $\lambda\in \hat{A}$. So, $K^{x^2} \geq K \cap K^x$ and then
   $$K \cap K^x \cap K^y \cap K^{xy} =K_G=\ker(\alpha ^{G})  = 1 \, .$$ 
Set  $B=K\cap K^x$; so, $B$ is $P$-invariant and  $B \cap B^y = 1$. 
By  \cite[V.6.4]{huppertI} we have that  $B^{\perp }=K^\perp (K^x)^\perp = \la \alpha, \alpha^x\ra$
has rank at most 2,  
so $B^{\perp}=\la \alpha \ra \times \la \alpha^x\ra$, as $\cent{\hat{A}}{P}=1_A$ (and $\rank(B^{\perp}) \neq 1$ as a cyclic $2$-group has no automorphism of order $3$). 
  Moreover,
 $$\hat{A}= 1^\perp = (B \cap B^y)^\perp = B^\perp B^{\perp y}=\la \alpha,\alpha^x\ra \la \alpha^y,\alpha^{xy} \ra = \la \alpha,\alpha^x\ra \la [\alpha,y],[\alpha^x,y]\ra=B^\perp [\hat{A},y].$$
In particular, $\rank(A) =\rank(\hat{A}) \leq 4$ and $\exp(A) = \exp(\hat{A}) = o(\alpha)$.
Write $e=\exp (A)$.
As $A\neq 1$, $\rank(A)$ is either $2$ or $4$ by Remark \ref{3.3a}.
If $\operatorname{rank}(A)=2$, then by Remark~\ref{3.3a}  $A=\langle a \rangle \times \langle a^{x}\rangle=\langle a\rangle^{G}$,
and we are done.
Thus, we can also assume that $B\ne 1$, since  if $B =1$, then $\hat{A} = 1^{\perp} = B^{\perp} =
\la \alpha \ra \la \alpha^x\ra$, so  $\rank(A) = \rank(\hat{A}) = 2$. 
Hence, $\operatorname{rank}(A)=4$ by Remark~\ref{3.3b}.

Let $Z =\z Q$. 
As  $H\cong C_6$ and $c(Q)  >2$, $Q \nor G$ and by Remark~\ref{3.6}
we have $\exp(A/Z)=4$ and $\class(Q)=3$. In particular, $e\ge 4$.
Observe also that $B\cap Z = 1$, since $B \cap Z \nor G$ and $B \cap B^y = 1$.

By Lemma \ref{pgroup}, $[\hat{A}, y] \cong [A,y] \cong A/Z$ and, as $\hat{A} = B^\perp B^{\perp y}$,
we see that $[\hat{A}, y] = [B^\perp, y][B^{\perp y}, y] = [B^\perp, y] = \la [\alpha, y], [\alpha, y]^x\ra$ has rank two.
Since $[\hat{A}, y]$ is a nontrivial $P$-module,
we conclude that 
$Q' = [A, y] \cong [\hat{A}, y] \cong A/Z \cong (C_4)^2$.

By Lemma~\ref{pgroup}  we also know that $y$ acts as the inversion on both $Q'$ and $A/Z$. It follows that
$Q' \cap Z = \Omega_1(Q')$ and  $A/{\bf Z}_2(Q) \cong (C_2)^2$.

As $B \cong BZ/Z$ and $BZ/Z$ is a nontrivial $P$-submodule  of $A/Z \cong (C_4)^2$, we have two cases: 
 
  \smallskip
  (a):
$B \cong (C_4)^2$ and   $A = B \times Z$.
  Since $B \cong (C_4)^2$, $B\times B^y \leq A\subseteq\N G$.
  Observe that $G=\cent{G}{BB^y}(BB^yH)$, we have that
  $B\times B^y \subseteq \N {BB^yH}$ by Lemma \ref{nonvanishingsmallgroup}.
  Now, applying part (2) of Lemma~\ref{XYtypeelementary} to $BB^yH$, we get a contradiction.

\smallskip
  (b): 
  $B \cong (C_2)^2$. Note that in this case  $BZ/Z = \cent{A/Z}y$, so $BZ = {\bf Z}_2(Q)$.
  We also observe that $Z \not\leq Q'$.
  In fact, if $Z \leq Q'$, then $Z = \Omega_1(Q')$, so   $B \cap Q' = 1$;  this
  implies that $Q'/Z$ is a complement of $BZ/Z$ in $A/Z \cong (C_4)^2$, a contradiction.
  
Let $\o A = A/B$.
Take $a \in A-BZ$.
Recalling  that 
$\o A \cong \widehat{A/B} \cong B^{\perp} \cong (C_e)^2$, where $e=\exp (A)$,  
we see that $\o{a} \not\in \o Z=\o{A}^2$ and hence that $o(a) = e$. 
Let $C=\la a,a^x\ra$. Since  $[A:C] = |A|/|B^{\perp}| = |B|= 2^2$,
 $C$ is a maximal $P$-submodule of $A$.
  We also observe that $e \geq 8$, as $e \leq 4$ implies $|A| \leq 8^2$, so $|Z| = 2^2$ and $Z$ is a minimal $P$-submodule of $A$,
  which is impossible as  $Z \not\leq Q'$.

We claim that  $C^{y}\neq C$. 
  Working by contradiction, assume that $C^y = C$, so $C \nor G$. 
  As $\cent{A/C}y$ is a nontrivial $P$-submodule of the irreducible $P$-module $A/C$, we deduce that $Q/C$ is abelian and hence that $Q' \leq C$.
  We have observed above that  $Z \not\leq Q'$; moreover, $Q' \not\leq Z$ because $\class(Q) > 2$, so 
  $Z$ cannot be contained in the  uniserial $P$-module $C$.
 By the maximality of $C$, it follows that  $A = CZ$. 
 So, $C/(Z \cap C) \cong A/Z  \cong (C_4)^2$ 
 and, as $Z \cap C = \Omega_1(Q') = \Omega_1(C) $,   we conclude that $e = 8$ and $C \cong (C_8)^2$. 
 Note that $c(C \langle y\rangle)=3$,  as $A = CZ$ implies  $Q' = [A, y] = [C, y]$.   
 Since by assumption $C \subseteq A \subseteq \N G$, an application of Lemma~\ref{nonvanishingsmallgroup} and
 Lemma~\ref{Xtypeclassification} to the group $G_0= CH$ yields that $aa^y=a^{4x}$ for a suitable generator $x$ of $P$.
 As $B\cap Z=1$, then $\rank(Z) < \rank(A) = 4$ and hence  $Z$ has rank $2$ by Remark~\ref{3.3b}.
 Since $|Z| = |A/C||C\cap Z| = 4^2$, we deduce that  $Z \cong (C_4)^2$.
 Let now $z\in Z$ be such that $z^2=a^4$ and consider $g=a^{x^{-1}}a^2z$.
 Then $o(g)=8$ and 
$  gg^y= a^{x^{-1}}a^2z(a^{x^{-1}}a^2z)^y=(aa^y)^{x^{-1}}(aa^y)^2 z^2 =a^4a^4=1$, so $g^y = g^{-1}$.
 Let $\lambda \in \irr{A}$ be such that $\lambda(g)=\zeta_8$ and $\lambda(g^{x})=\zeta_8^{5}$: 
 consider, for instance, an extension to $A$ of a suitable character of $\langle g\rangle \times \langle g^{x}\rangle$. 
 Since $gg^{x}g^{x^{2}}=1$, $\lambda(g^{x^{2}})=\lambda(gg^{x})^{-1}=\zeta_8^{2}$.
 As $g^{y}=g^{-1}$, we have
 \[
 \lambda^G(g)=\sum\limits_{0\leq i\leq 2}(\lambda^{x^i}(g)+\lambda^{x^iy}(g))= \sum\limits_{0\leq i\leq 2}\lambda^{x^i}(g)+\sum\limits_{0\leq i\leq 2}\o{\lambda^{x^i}(g)}=\sum_{j\in \{ 1,2,3,5,6,7 \}} \zeta_8^{j}=0.
 \]
 Hence, by Lemma \ref{ind} we conclude that $g \in \V G$, against the assumption $A \subseteq \N G$.

 Hence $C^{y}\neq C$, and then $A=CC^{y}=\langle a\rangle^{H}$ by the maximality of $C$, so $A$ is a cyclic $H$-module.

\end{proof}

The next result is also key to classifying the  finite groups $G$ such that $\P G=\frac{5}{6}$.

\begin{pro}\label{C6Hindecomp}
  Assume Setting \ref{hyp1} and that $A \subseteq \N G$.
  If $A$ is an indecomposable $H$-module, then $A$ is a cyclic  $H$-module.
\end{pro}
 \begin{proof}
By Remark~\ref{3.6}, $c(Q) = 3$ and, since 
   $$Q\lesssim \bigtimes_{\lambda\in\irr{A}-\{ 1_A \}}Q/\ker(\lambda^G),$$
    there exists a
   nonprincipal $\alpha\in\irr{A}$ such that $c(Q/K)=3$,  where $K=\ker(\alpha^G)$.
   As $A \subseteq \N G$, $A/K \subseteq \N {G/K}$ and 
   by Lemma \ref{kerfreeA}, $A/K$ is a cyclic  $H$-module. So, $A= A_0K$ where  $A_0=\la a_0\ra^H$ is the cyclic $H$-module generated by a suitable element $a_0 \in A$.
   Let  $N=[A_0,y,y]$.
   Since $c(A_0 \la y\ra  )= c(Q/K) = 3$, an application of Lemma~\ref{nonvanishingsmallgroup} and  of Proposition \ref{classcyclicmodule} to $G_0=A_0 \rtimes H$ yields that 
   $N$ is a unique irreducible $H$-submodule of $A_0$.
   As 
   $$N/(N\cap K)\cong NK/K=[A_0K/K,y,y]=[A/K,y,y]>1,$$
   we deduce that $N\cap K=1$, 
   and by the uniqueness of $N$ we conclude that $A_0 \cap K = 1$, so $A=A_0 \times K$. 
   Since $A$ is an indecomposable $H$-module, then $K=1$ and $A=A_0$ is a cyclic $H$-module. 
   \end{proof}

The next three results will conclude the description of the groups $G$ satisfying Setting \ref{hyp1} and  such that $\P G=\frac{5}{6}$.

\begin{lem}\label{C6combelem}
  Assume Setting \ref{hyp1} and let $Y=\la y\ra$.
  Suppose that $A=B\times C$,  where $B=\la b\ra^H$ and $C=\la c\ra^H$, and let $D=\langle bc\rangle^{H}$.
  If $A \subseteq \N G$, then
  \begin{description}
    \item[(1)] $c(DY)=c(Q)=\max \{ c(BY),c(CY) \} = 3$ and $\exp (D)=\max \{ \exp (B),\exp (C) \}$. 
    \item[(2)] $\exp(B)\geq \exp(C)$ if and only if $c(BY)\geq c(CY)$. In particular, $\exp(B)=\exp(C)$ if and only if $c(BY)= c(CY)$.
    \item[(3)] $\operatorname{rank}(D)=\max \{ \operatorname{rank}(B), \operatorname{rank}(C) \}$.
  \end{description}
\end{lem}
\begin{proof}
  By Remark~\ref{3.6}, $c(Q)=3$.
  
 (1) As $A=B\times C$ where $B$, $C$ are $G$-invariant subgroups of $A$ and $Q=AY$ where $Y=\langle y\rangle$, 
 $$\gamma_{i+1}(Q)=[A,\underbrace{Y,\cdots,Y}_i]=[B,\underbrace{Y,\cdots,Y}_i]\times [C,\underbrace{Y,\cdots,Y}_i]$$
 for $i\ge 0$.
 Therefore, $\max \{ c(BY),c(CY) \}=c(AY)=c(Q)=3$.
 Without loss of generality, we may assume that $c(BY)=3$.
 To show that $c(DY)=3$, since $c(DY)\le c(Q)=3$, it suffices to show that $[bc,y,y]\ne 1$, so  $[D,Y,Y]\ne 1$.
 An application of part (1) of Lemma \ref{pgroup} yields $[bc,y,y]=[b,y,y][c,y,y]$.
 As $A \subseteq \N G$, by Lemma \ref{nonvanishingsmallgroup} $B \subseteq \N {BH}$  and 
 the last statement of Proposition \ref{classcyclicmodule} applied to $BH$ gives $[b,y,y]\neq 1$.
 Therefore, as $[b,y,y]\in B$, $[c,y,y]\in C$ and  $B \cap C = 1$, we deduce that $[bc,y,y]\ne 1$.
 
 Finally, we observe that $o(b)=\exp (B)$ and $o(c)=\exp (C)$,
 so $\exp(D)=o(bc)=\max \{ o(b),o(c) \}=\max \{ \exp (B),\exp(C) \}$. 

 (2) Set $e=\max \{ \exp (B),\exp (C) \}$.
Then by (1)  $c(DY)=c(Q)=3$ and  $o(bc)=\exp(D) = e$.
 For $F=\langle f\rangle^H\leq A$, as $y$ acts as the inversion of $[F,y]$ (by Lemma~\ref{pgroup}), $c(FY)=3$ if and only if $[f,y]^2\ne 1$.

 Assume, working by contradiction, that $e = \exp(B)\geq \exp(C)$ and $c(BY)<c(CY)=3$.
 So, $[b,y]^2=1$.
 Observe that  $(bc)^{\frac{e}{2}} \notin C$. 
 Since $[bc,y]^{2}=[b,y]^{2}[c,y]^{2} = [c,y]^{2}\ne 1$ (as $c(CY)=3$) and $B \cap C = 1$, 
 \[
  (bc)^{\frac{e}{2}}=b^{\frac{e}{2}}c^{\frac{e}{2}}
  \notin \la [c,y]^2,[c,y]^{2x}\ra=\la [bc,y]^2,[bc,y]^{2x}\ra.
\]
This implies, in particular, that $\rank(D) = 4$ and,  
 applying Proposition \ref{classcyclicmodule} to $DH$, we deduce  that $bc\in\V{DH}$.
 Hence,  $bc\in\V G$ by Lemma \ref{nonvanishingsmallgroup}, a contradiction.

 Assume now, again working by contradiction, that $3=c(BY)\geq c(CY)$ and $\exp(B)<\exp(C)$. 
Since $[bc,y]^{2}=[b,y]^{2}[c,y]^{2}$ and  $[b,y]^{2}\ne 1$ (as $c(BY)=3$), we have
\[
  (bc)^{\frac{e}{2}}=b^{\frac{e}{2}}c^{\frac{e}{2}}=
  c^{\frac{e}{2}}\notin \la [b,y]^2,[b,y]^{2x}\ra=\la [bc,y]^2,[bc,y]^{2x}\ra 
\]
which as above  implies $bc\in \V G$, a contradiction.

(3) We first show that $\operatorname{rank}(D)=4$ provided that either $B$ or $C$ has rank $4$.
To see this, we may assume that $\rank(B)=4$. 
If $\rank(D)=2$, then $D = \la bc, (bc)^x \ra$ and, for suitable integers $i$ and $j$, 
\[
[b,y][c,y]=[bc,y]=(bc)^i(bc)^{jx}=b^ib^{jx}(c^ic^{jx}).
\]
As $A=B\times C$, where $B$ and $C$ are $G$-invariant subgroups of $A$,
$[b,y]=b^ib^{jx}$ (and $[c,y]=c^ic^{jx}$). So $b^y \in \la b, b^x\ra$ and hence $B = \la b, b^x\ra$ has rank $2$,  a contradiction.

We next show that $\operatorname{rank}(D)=2$ provided that both $B$ and $C$ have rank $2$.
In fact, otherwise  $\operatorname{rank}(D)=4$ by Remark~\ref{3.3b}.
Without loss of generality, we may also assume that $c(CY)\le c(BY)=3$.
Hence, by (2), $\exp (B)\ge \exp (C)$.
As $A \subseteq \N G$, 
an application of Lemma \ref{nonvanishingsmallgroup} to both $HB$ and $HC$ yields that $B \subseteq \N{HB}$ and $C \subseteq  \N {HC}$.
Since $\operatorname{rank}(B)=2$, it follows by Proposition \ref{classcyclicmodule} that $bb^{y}=b^{4x}$ (for a suitable generator $x$ of $P$) and $o(b)=\exp(B)=8$.
As $o(c)=\exp(C)\le \exp(B)=o(b)$, then  $o(bc)=\exp(D)=8$.
Now, applying Lemma \ref{nonvanishingsmallgroup} to $HD$, we have that $D \subseteq \N {HD}$.
Since $\operatorname{rank}(D)=4$, Proposition \ref{classcyclicmodule} implies that $[(bc)^{2},y]=(bc)^{4x^{i}}$ where $i=1$ or $-1$.
Observing that 
\[
  [b^{2},y][c^{2},y]=[(bc)^{2},y]=(bc)^{4x^i}=b^{4x^{i}}c^{4x^{i}}  
\]
and that $A=B\times C$ where $B$ and $C$ are $G$-invariant, we deduce that  $b^{-2}(b^{2})^{y}=[b^{2},y]=b^{4x^{i}}$.
As $bb^{y}=b^{4x}$, then  $b^{2}(b^{2})^y=(bb^{y})^{2}=b^{8x}=1$ where the last equality holds as $o(b)=8$.
Thus,
$[b^2, y] = [b, y]^2 = (b^{-2}bb^y)^2 = b^{-4} = b^4$, so $b^4 = b^{4x^i}$ with $i \in \{1, -1\}$, and hence
$1 \neq b^4 \in \cent AP$, a contradiction.
\end{proof}

\begin{lem}\label{C6combineHind}
  Assume Setting \ref{hyp1} and let $Y=\la y\ra$.
  Suppose that $A=B\times C$, where $B$ and $C$ are cyclic  $H$-modules such that $\exp(B)\geq \exp(C)$.
  If $A \subseteq \N G$, then one of the following happens.
\begin{description}
  \item[(1)] $c(BY)=c(CY)=3$,  and $B$ and $C$ are isomorphic $H$-modules.
  \item[(2)] $c(BY)=3>c(CY)$, $\rank(B)\geq\rank(C)$ and $\exp(B)>\exp(C)$. 
    In addition, if $\rank(B)=2$, then $y$ acts as the inversion on $C$. 
\end{description}
\end{lem} 
\begin{proof}
  Let $B=\la b\ra^G$,  $C=\la c\ra^G$ and $D=\langle bc\rangle^{G}$. So, $e=\exp (B) = \exp (A)$, as $\exp(B)\ge \exp(C)$.
  Observe that 
  $c(Q)=3$ by Remark~\ref{3.6}.
By Lemma \ref{C6combelem},  $3=c(DY)=c(BY)\ge c(CY)$, $\exp (D)=e$ and $\operatorname{rank}(D)=\max \{ \operatorname{rank}(B),\operatorname{rank}(C) \}$.
  Also, since $A \subseteq \N G$, an application of Lemma \ref{nonvanishingsmallgroup} to $HB$, $HC$ and $HD$ yields that $B \subseteq \N {HB}$, $C \subseteq \N {HC}$ and $D \subseteq \N {HD}$.

  We show first that if $c(BY)=3$ and $bb^y =  b^{\frac{e}{2}x^{i}}$ for $i\in\{-1,1\}$, then $[b^{2},y] \neq b^{\frac{e}{2}x}$ .
  In fact,  $bb^y=b^{\frac{e}{2}x^{i}}$ implies that  $\operatorname{rank}(B)=2$.  
  As $BY$ satisfies Setting \ref{hyp1} and $B \subseteq \N {HB}$, Proposition \ref{classcyclicmodule} yields $e=8$ and 
  $bb^y = b^{4x^i}$. So, $[b,y] = b^{-2} b^{4x^i}$ and $[b^2, y]  = [b, y]^2= b^{-4} = b^4 \neq b^{4x}$, as $1 \neq b^4 \not\in \cent AP$.

  \medskip
 (1) Suppose that $c(BY)=c(CY)=3$. 
  So,  $\exp(D)=\exp(C)=\exp(B)=e$ by Lemma~\ref{C6combelem}.
 We first show that  $\rank(B)=\rank(C)$.
 Working by contradiction, we may assume that $\rank(B)=4$, so $\rank(D)=4$,  
 and that $\rank(C)=2$.
 As $D \subseteq \N {HD}$ and $C \subseteq \N {HC}$ and both $HD$ and $HC$ satisfy  Setting~\ref{hyp1},
 Proposition \ref{classcyclicmodule} yields that $e = 8$ and,
 for a suitable generator $x$ of $P$, $[b^2, y] [c^2, y] = [(bc)^{2},y]=(bc)^{4x} = b^{4x} c^{4x}$, so $[c^2, y] = c^{4x}$,
 and  also that   $cc^y=c^{4x^i}$ for $i\in\{-1,1\}$, against the second paragraph of this proof.

 Next, we show that $B$ and $C$ are  isomorphic $H$-modules. 
 Assume first that $\operatorname{rank}(B)=\operatorname{rank}(C)=2$.
 Then $\operatorname{rank}(D)=2$.
Since $D\subseteq\N {HD}$, an application of Proposition \ref{classcyclicmodule} to $HD$
yields that $e=8$ and, for a suitable generator $x$ of $P$, we may assume that $(bc)(bc)^{y}=(bc)^{4x}$.
Thus, $bb^y=b^{4x}$ and $cc^y=c^{4x}$. 
Defining   $\varphi(b)=c$, $\varphi(b^{x})=c^{x}$,  
it is routine to check that $\varphi$ extends to an isomorphism of  $H$-modules from $B$ to $C$.

 Assume now  that $\operatorname{rank}(B)=\operatorname{rank}(C)=4$.
 Then $\operatorname{rank}(D)=4$.
 Since $D\subseteq\N {HD}$, an application of Proposition \ref{classcyclicmodule} to $HD$
yields that, for a suitable generator $x$ of $P$,  $[(bc)^{2},y]=(bc)^{\frac{e}{2}x}$.
Hence, as above $[b^{2},y]=b^{\frac{e}{2}x}$ and $[c^{2},y]=c^{\frac{e}{2}x}$.
 Let $b_0=wb^{\frac{e}{4}x}$,  where $w=[b,y]$; so,  $w =b^{-\frac{e}{4}x}b_0$.
 Then $o(b_0)=2$, as $b_0^2 = 1$ and $b_0 \neq 1$ by Remark \ref{3.3b}. 
 We show that $[b_0,y]=b^{-\frac{e}{2}x}b^{-\frac{e^{2}}{16}x^{-1}}$ and  that 
 $B=\langle b,b^{x}\rangle \times \langle b_0,b_0^{x}\rangle$.
 
In fact, we have 
 \[
  [b_0,y]=[wb^{\frac{e}{4}x},y]=[w,y] [b^{\frac{e}{4}x},y]=w^{-2}w^{\frac{e}{4}x}=b^{-\frac{e}{2}x}b^{-\frac{e^{2}}{16}x^{2}}.
 \]
 where the third equality holds as $[w,y]=w^{-2}$ (as $y$ inverts $[B, y]$),  and the last equality holds as
 $\frac{e}{4}\ge 2 = o(b_0)$.
 In particular, it follows that  $[b_0,y]\ne 1$.
 As $\Omega_1(\langle b,b^{x}\rangle)=\langle w^{2},w^{2x}\rangle$, $y$ acts trivially on $\Omega_1(\langle b,b^{x}\rangle)$ by part (1) of Lemma \ref{pgroup} and hence, as $[b_0,y]\ne 1$ and  $\langle b_0,b_0^{x}\rangle$ is an irreducible $P$-module, 
 we deduce that  $\langle b_0,b_0^{x}\rangle \cap \Omega_1(\langle b,b^{x}\rangle)=1$. 
Since $B = \la b, b^x, b^y, b^{xy} \ra = \la b, b^x, b_0, b_0^x \ra$, we get  $B=\langle b,b^{x}\rangle \times \langle b_0,b_0^{x}\rangle$.
 Similarly, $C=\langle c,c^{x}\rangle\times \langle c_0,c_0^x\rangle$, with  $[c,y]=c^{-\frac{e}{4}x}c_0$ and $[c_0,y]=c^{-\frac{e}{2}x}c^{-\frac{e^{2}}{16}x^{-1}}$,  where $c_0=[c,y]c^{\frac{e}{4}x}$.
 Defining  $\varphi(b)=c$, $\varphi(b^{x})=c^{x}$, $\varphi(b_0)=c_0$ and $\varphi(b_0^{x})=c_0^{x}$, one readily checks  
 that $\varphi$ extends to  an isomorphism of  $H$-modules from $B$ to $C$.

\smallskip
 (2) Suppose now that  $c(BY)=3>c(CY)$.
 Then by Lemma~\ref{C6combelem} $e=\exp(D)=\exp(B)>\exp(C)$ and $c(DY)=3$.
 Next, we show that $\rank(B)\geq \rank(C)$.
 In fact, if $\rank(B) < \rank(C)$ then $\rank(B)=2$ and $\rank(C)=4$; so  $\rank(D)=4$.
 Since $HD$ satisfies Setting \ref{hyp1} and  $D=\langle bc\rangle^{G} \subseteq \N {HD}$, 
 Proposition \ref{classcyclicmodule} yields that, for a suitable generator $x$ of $P$, 
 $[(bc)^{2},y]=(bc)^{\frac{e}{2}x}$.
 Similarly, applying Proposition \ref{classcyclicmodule} to $HB$, we have that ($e = 8$ and) $bb^y=b^{\frac{e}{2}x^i}$ for $i\in\{-1,1\}$.
 So, $[b^2,y]=b^{\frac{e}{2}x}$ and $bb^y=b^{\frac{e}{2}x^i}$ for $i\in\{-1,1\}$, against the second paragraph of this proof.

 Finally, we assume that  $\rank(B)=2$; then $\rank(C) \leq 2$ and $\rank(D)=2$.
 Then, using again  Proposition \ref{classcyclicmodule}, $o(bc)=8$ and  for a suitable generator $x$ of $P$  
 \[
  bc(bc)^y=(bc)^{4x}. 
  \]
  Thus $cc^y = c^{4x}$ and, since $o(c)<o(b)=8$, we deduce that  $c^y=c^{-1}$. Then, it easily follows that $y$ acts as the inversion on
  $C$. 
\end{proof}

\begin{pro}\label{C6classify}
  Assume Setting \ref{hyp1}.
  Then $A\subseteq \N G$ if and only if $A=B\times C$ where $B$ and  $C$ are normal subgroups of  $G$ such that $\exp (B)>\exp (C)$, 
   $[C,y,y]=1$,   and either  $y$ acts as the inversion on $C$ and $B$ is a homogeneous   $H$-module of type {\rm (T1)},  or $B$ is a homogeneous $H$-module of type \rm{(T2)}.
 \end{pro}
 

\begin{proof}
  We  write $A=B\times C$ where $B$ is a direct product of  indecomposable $H$-modules,
  say $B_i$ for $1\le i\le k$, such that $c(B_i \la y\ra)>2$  and $C$ is a $H$-module such that $[C,y,y]=1$.
  Note that by Lemma \ref{nonvanishingsmallgroup} and Remark~\ref{3.6},
  $c(B_i\la y \ra) = 3$ for every $1 \leq i \leq k$.  
  By Proposition \ref{C6Hindecomp}, Lemma \ref{C6combineHind} and Proposition \ref{classcyclicmodule}, we can assume that $k\geq 2$.

  \smallskip
 
  Suppose first that $A\subseteq \N G$. 
For $i\in \{ 2,\ldots,k \}$, let $S_i=H(B_1\times B_i)$.
  Observe that $G=\cent{G}{B_1\times B_i} S_i$, and hence an application of Lemma \ref{nonvanishingsmallgroup} yields $B_1 \times B_i\subseteq \N {S_i}$.
  As $S_i$ satisfies Setting \ref{hyp1} and $B_1$ and $B_i$ are indecomposable $H$-modules,  
  by  Proposition \ref{C6Hindecomp}, Lemma \ref{C6combineHind} and Proposition \ref{classcyclicmodule} we deduce, using again Lemma~\ref{nonvanishingsmallgroup}, 
  that the $B_j=\la b_j\ra^G$, for $1 \leq j \leq k$,  are pairwise isomorphic $H$-modules of type (T1) or (T2).
In the same way, applying Lemma~\ref{C6combelem} to $(B_1 \times \la c \ra^G)H$, for $c \in C$, we have that   
$\exp (B)>\exp (C)$.
Finally, we assume that $B$ is direct product of isomorphic $H$-modules of type (T1) and we consider an element $c \in C$. 
Then part (2) of Lemma~\ref{C6combineHind} applied to $(B_1 \times \la c \ra^G)H$ yields $c^y = c^{-1}$. So, $y$ acts as the inversion on $C$. 
 
  \smallskip
  Conversely, we assume that $A = B \times C$, where $B = B_1 \times \cdots \times B_k$ with $B_i$ isomorphic $H$-modules of
  type (T1) or of type  (T2), and in the first case $y$ acts as the inversion on $C$,   for
  $1 \leq i \leq k$, $\exp(B) > \exp(C)$  and $[C, y,y] = 1$.  
  We prove that $A \subseteq \N G$.  
  Let $a_{0}\in A$, $A_{0}=\la a_{0}\ra^G$ and $T= A_{0} H$. 
  Write $a_{0}=b_{0}c_{0}$, where $b_{0}\in B$ and $c_{0}\in C$, and set $B_{0}=\langle b_{0}\rangle^{G}$, $C_{0}=\langle c_{0}\rangle^{G}$.
  Therefore, $A_{0}\le B_{0}\times C_{0}$.
  By Lemma \ref{nonvanishingsmallgroup}, it suffices to show $a_{0}\in \N T$.
  Also,  by part (2) of Lemma \ref{5/6} we can assume $c(A_{0} \la y\ra)=3$.
  As $c(C_{0}  \la y\ra)\leq c(C \la y\ra)\leq 2$ and $(B_{0} \times C_{0})H$ satisfies Setting~\ref{hyp1}, 
  by Lemma \ref{C6combelem} $c(B_{0}\la y\ra)=3$.
  
  Suppose first that the modules $B_i$ are of type (T1). 
  Then $gg^{y}=g^{4x}$ for all $g \in B_i$, $1 \leq i \leq k$.
  Let $b \in B_{0}$ and write $b=\prod_{i=1}^kg_i$,  where $g_i\in B_i$.
  Then 
  \[
  bb^{y}=\prod_{i=1}^{k}g_i(\prod_{i=1}^{k}g_i)^y=\prod_{i=1}^{k} (g_i g_i^y)=\prod_{i=1}^{k} g_i^{4x}=(\prod_{i=1}^{k} g_i)^{4x}=b^{4x}.
  \]
  This implies that  $\rank(B_{0})=2$ and, since $c(B_{0} \la y\ra)=3$ and $\exp(B)=8$, also  that $\exp(B_{0}) =8$.
Let now $a \in A_{0}$ and write $a = bc$ with $b \in B_{0}$ and $c \in C_{0}$. 
Recalling that $y$ acts as the inversion on $C$ and that $\exp(C) < \exp(B) = 8$, we have 
  \begin{center}
    $aa^{y}=bc(bc)^{y}=(bb^{y})(cc^{y})=b^{4x}=b^{4x}c^{4x}=(bc)^{4x}=a^{4x}$.
  \end{center}
  This implies that  $\rank(A_{0})=2$ and, as $c(A_{0}\la y \ra) = 3$, that $\exp(A_{0}) \geq 8$; so
  $\exp(A_{0}) = \exp(B) = 8$. 
  Since $T = A_0H$ satisfies Setting~\ref{hyp1}, an application of Proposition \ref{classcyclicmodule} to $T$ yields $a_{0}\in\N {T}$, as desired.

  Suppose now that the modules $B_i$ are of type (T2). 
  Then $\exp(B) = 2^n \geq 8$ and  $g^{2^{n-1}}=[g^{2},y]^{x}$ for all $g \in B_i$.
Let $a \in A_{0}$ and write $a = bc$ with $b \in B_{0}$, $b=\prod_{i=1}^kb_i$ where $b_i\in B_i$, and $c \in C_{0}$.   
Observe that $o(c)<\exp(B)=2^n$ and that  $o([c,y])\leq \exp([C,y])\leq 2$ (recall that $[C,y,y]=1$ and that $y$ acts as the inversion on $[C,y]$ by Lemma \ref{pgroup} if $[C,y]\neq 1$).
So
 \begin{equation}\label{ytype}
    [a^{2},y]^x=[(c\prod_{i=1}^kb_i)^2,y]^x=[c^2,y]^x\prod_{i=1}^k[b_i^2,y]^x=\prod_{i=1}^k b_i^{2^{n-1}}=c^{2^{n-1}}(\prod_{i=1}^kb_i)^{2^{n-1}}=a^{2^{n-1}}.
   \end{equation} 
   Since  $y$ acts as the inversion on $[A_{0},y]$ (by part (1) of Lemma \ref{pgroup}), then
   $[A_{0}, y] = \la [a_{0}, y], [a_{0}^x, y] \ra$ has rank $2$ 
   and, recalling that   $A_0$ is a $P$-module and that $c(A_{0} \langle y\rangle)=3$, we deduce that  $[A_{0},y]\cong (C_4)^{2}$.
   Thus,  (\ref{ytype}) implies that $\exp(A_{0}) =2^{n} = \exp(A)$.
   We next show that $\rank(A_{0}) = 4$. 
Assuming the contrary, then $A_{0}$ has rank $2$, so it is a  uniserial  $P$-module.
   As $c(A_{0} \langle y\rangle)=3$ (and both $\cent{A_{0}}{y}$ and $[A_{0},y]$ are $P$-submodules of $A_{0}$), 
then $\cent{A_{0}}{y}\leq [A_{0},y]$.
   Since by Lemma \ref{pgroup} $A_{0}/\cent{A_{0}}{y}\cong[A_{0},y]$, we deduce that  $A_{0} \cong (C_8)^{2}$.
   So, $\exp(A) = \exp(A_{0}) =  8$.
As above, we write  $a_{0} = b_{0}c_{0}$, with $b_{0} \in B$, $b_{0} =\prod_{i=1}^{k}g_i$, $g_i \in B_i$,   and $c_{0} \in C$.  
 As $A=(\bigtimes_{i=1}^k B_i)\times C$ and  $[a_{0},y]\in \la a_{0},a_{0}^x\ra$, we have  $[g_i,y]\in \la g_i,g_i^x\ra$ for $1\leq i\leq k$.
   Also, since  $c(A_{0} \la y\ra)=3$, there exists an index $d$,  $1\leq d\leq k$, such that $c(\la g_{d}\ra^G \la y\ra)=3$ by Lemma \ref{C6combelem}.
   Let  $D =\langle g_{d}\rangle^{G}$.
   Since $c(D\langle y\rangle)=3$ and $\operatorname{rank}(D)=2$,
   the arguments used above similarly yield  that $D\cong (C_8)^{2}$ and $[D,y]\cong (C_4)^{2}$. 
    Since $B_d$ is of type (T2) and exponent $8$, it easily follows that $|B| = |D|\cdot 2^2$, so $D$ is a maximal $P$- submodule of $B_{d}$. 
Writing  $B_{d}=\langle b_{d}\rangle^{G}$, we observe that $b_d \not\in D$, as $D$ is $H$-invariant, so   $B_{d}=D \langle b_{d},b_{d}^{x}\rangle$.
Thus, $D \cap \langle b_{d}, b_{d}^{x}\rangle$ is a $P$-submodule  of $D$ of order $16$,
   and hence $D \cap \langle b_{d}, b_{d}^{x}\rangle=[D,y]$.
   Since $[D,y]\le [B_{d},y]\cong (C_4)^2$, we conclude that $[B_{d},y] = [D, y] \le \langle b_{d}, b_{d}^{x}\rangle$, 
giving $\rank(B_d) = 2$, a contradiction.
   
Thus, $\rank(A_{0}) = 4$. As $A_{0}=\la a_{0},a_{0}^x\ra \la [a_{0},y],[a_{0},y]^x\ra$ and for every $a \in A_0$
$[a^{2},y]^x= a^{2^{n-1}}$, where $2^n = \exp(A_0)$,   
  Proposition \ref{classcyclicmodule} yields $a_{0}\in \N T$.
\end{proof}

 \begin{rem}
\normalfont We remark that, using the notation of Proposition \ref{C6classify}, if $y$ acts as the inversion on $C$ and $[C,y,y]=1$, then $\exp(C)$ divides $4$. 
 \end{rem}

\section{Proof of the main theorem}
%
We recall that, given a prime number $p$,  an element $g$ of a finite group $G$ can be uniquely written in the form
$g = g_p g_{p'}$,  where $g_p$ is a $p$-element, $g_{p'}$ is a $p'$-element and
$g_p$ and $g_{p'}$ commute.
The element $g_p$ is called the \emph{$p$-part of $g$}; we will use this notation below. 

\begin{lem}[\mbox{\cite[Proposition 3.2]{ZYD2022}.}]\label{56connect}
  Let $A$ be an abelian normal subgroup of the group $G$ such  that $[G:A] \leq 6$ and $[G:A] \neq 5$.
  Then, for $a \in A$:
  \begin{description}
  \item[(1)] if $[G:A] \neq 3$ and $G$ has  abelian Sylow $3$-subgroups, then $a \in \V G$ if and only if $a_2 \in \V G$;
      \item[(2)] if $[G:A] = 3$, then  $a \in \V G$ if and only if $a_3 \in \V G$.
  \end{description}
\end{lem}


\begin{lem}\label{p>2abel}
 Let $G$ be a finite group such that  $\P G < \mathfrak{a}$. 
 Then for every odd prime $p$, the Sylow $p$-subgroups of $G$ are abelian. 
\end{lem}
\begin{proof}
  Let $G$ be a minimal counterexample and $P$ a  nonabelian Sylow $p$-subgroup of $G$ for some odd prime $p$.
  Recalling that $\P{\o G} \leq \P G$ for every factor group $\o G$ of $G$, the minimality of $G$ implies that $\mathbf{O}_{p'}(G) = 1$.  
Since $\P G < \mathfrak{a}$, by Theorem \ref{zyd} $A=\N G$ is an abelian normal subgroup of $G$. 
  Hence, $\P G = 1 -\frac{1}{[G:A]} < \mathfrak{a}$ implies $[G:A]=m\leq 6$.
  So, $m= 6$, $p = 3$  and $A \leq P$, so $[G:P] = 2$ and $[P:A]=3$.  Using again the minimality of $G$, we deduce that $P'$ is a minimal normal subgroup of $G$.
  It follows that $P' \leq \z P$ and hence  (as $[G:P] = 2$) $|P'| = 3$. So, by Problem 2.13 and Problem 6.3 of \cite{isaacs1976} every nonlinear
  irreducible character of $P$ vanishes on $P - \z P$. Clifford theory hence yields that every $\chi \in \irr G$ such that $P' \not\leq \ker{\chi}$
  vanishes on $P - \z P$, so $A = \z P$ and $[P:A]\geq 3^2$, a contradiction.  
\end{proof}

%
%

\begin{lem}\label{Agroups}
  Let $G$ be an $\mathcal A$-group; let  $Z = \z G$ and $F = \fitt G$. Then 
  \begin{description}
\item[(1)] $\fitt{G/Z} = F/Z$ and $\z{G/Z} = 1$.
  \item[(2)] If $\P G < \mathfrak{a}$, then $\N G = F$ and $\P{G/Z} = \P G$.
  \item[(3)] If $[G:F] = 5$, then  $\P G < \mathfrak{a}$ if and only if $[F:Z]$ is not divisible by $6$.
  \end{description}
\end{lem}
\begin{proof}
(1) As $G$ is a $\mathcal{A}$-group, then  $G' \cap Z = 1$ (\cite[VI.14.3(b)]{huppertI}).
Let $W/Z = \z{G/Z}$; then $[W, G] \leq Z \cap G' = 1$, so $W \leq Z$ and hence $W = Z$, i.e.  $\z{G/Z} = 1$.
Let $X/Z = \fitt{G/Z}$; $X/Z$ is abelian, so $[X, X] \leq G' \cap Z = 1$, thus $X$ is a normal abelian subgroup of $G$ and $X \leq F$.
Hence, $X = F$. 

\medskip
(2) First, we show that $\N G = F$. 
As $\P G < \mathfrak{a}$, by Theorem \ref{zyd}, $\N G$ is an abelian normal subgroup of $G$, so $\N G \leq F$.
If $G$ is nilpotent, than $G$ (being an $\mathcal{A}$-group) is abelian, and clearly $\N G = G = F$. 
If $m = [G:\N G]$ is a prime, then $G$ is nonabelian and $\N G = F$.
Finally, if $m  \in \{4, 6\}$, then  for every $a \in F$   Lemma~\ref{56connect} yields   $a \in \N G$, because  $a_2$ is a
nonvanishing element of $G$ by Lemma~\ref{brough}. Hence, $\N G = F$.

In order to show that $\P{G/Z} = \P G$, we observe that 
$G/Z$ is a $\mathcal{A}$-group and $\P{G/Z} \leq \P G < \mathfrak{a}$ so, by what we have just proved and  part (1),  we have $\N{G/Z} = F/Z$, hence
$\P{G/Z} = \P G$.

\medskip
(3) Assume $[G:F] = 5$.  By (1), $G/Z$ is a Frobenius group with abelian kernel $F/Z$. If $6$ divides $|F/Z|$, then there is a normal subgroup $N$ of  $G$,
with $Z \leq N$, such that $\o G = G/N$ is a Frobenius group with kernel $\o F = \o U \times \o V$, where $\o U$ is a $2$-group,
$\o V$ is a $3$-group and $\o U$ and $\o V$ are minimal normal subgroups of  $\o G$.
Then,  by Lemma~3.3 of~\cite{ZYD2022} $\N{\o G} = \o U \cup \o V$ and $\P G \geq \P{\o G} > \mathfrak{a}$.

Assume now that $6$ does not divide $[F:Z]$ and observe that $F = G' \times Z$ (\cite[VI.14.7(b)]{huppertI}). For every  $a \in G'$ and $\alpha \in \irr F$,
$\alpha^G(a)$ is a sum of five $|G'|$-th roots of unity, and $|G'|$ is coprime both to $5$ and to either $2$ or $3$.
Hence, $a \in \N G$ by Lemma~\ref{vs} and Lemma~\ref{ind}, so $F \subseteq \N G$ by Lemma \ref{nonvancent}, proving that
$\P G \leq \frac{4}{5} < \mathfrak{a}$.
\end{proof}

We are now ready to prove Theorem~A, which we state again.

\begin{thm}
  Let $G$ be a finite group, $Q\in\syl{2}{G}$ and $P\in\syl{3}{G}$.
%
  Then $\P G<\P{A_7}=\mathfrak{a}$ if and only if 
  \begin{description}
  \item[(a)] $G$ is an $\mathcal{A}$-group such that  $[G:\fitt{G}]=m\leq 6$,  and  
   $|\fitt{G}/\z{G}|$ is not divisible by $6$ if $m = 5$.
 \item[(b)]  $Q$ is nonabelian and $G$ has an abelian normal subgroup $A$ such that  
    $Q\cap A  = Z_0 \times D$,
    where $Z_0 = \cent{Q \cap A}P  \leq \z{G}$  and $D = [Q \cap A, P]$, 
  and  one of the following holds:
  \begin{description}
  \item[(b1)] $|G/A|=4$ and $Q \cap A=\z{Q}$. 
 \item[(b2)] $G/A\cong S_3$,  
   and, for some  $x\in P- A$,   $1 \neq D =Z\times Z^{x}$, where $Z=\cent{D}{Q}$ is either elementary abelian or isomorphic
   to $C_4\times (C_2)^t$, with $t \geq 0$. 
 \item[(b3)] $G/A\cong C_6$, 
   $1 \neq D = B\times C$ with $B$ and $C$ are normal subgroups of  $G$ such that  
   $[C,Q,Q]=1$   and,  if  $B \neq 1$, then  $\exp (B)>\exp (C)$ and either every $y \in Q -A$ acts as the inversion on $C$ and 
   $B$ is a homogeneous $G/A$-module of type {\rm (T1)},  or $B$ is a homogeneous $G/A$-module of type {\rm (T2)}. 
\end{description}
\end{description}
  \end{thm}
 \begin{proof}
  Suppose that $\P G<\mathfrak{a}$.
  By Theorem \ref{zyd}, $A=\N G$ is an abelian normal subgroup of $G$, so $\P G = 1-\frac{1}{[G:A]}$ and hence $[G:A]=m\leq 6$.
  Also,  for every odd prime number $p$,  $G$  has abelian Sylow $p$-subgroups by Lemma \ref{p>2abel}.
  Let $Q$ be a Sylow $2$-subgroup of $G$. 
  If $Q$ is abelian then $G$ is an $\mathcal{A}$-group, and $A = \fitt G$ by Lemma \ref{Agroups}.
  Hence, $[G:\fitt G] = m \leq 6$ and, recalling also part (3) of Lemma \ref{Agroups}, we have  case  (a). 

  We can hence assume that $Q$ is nonabelian, 
 so $m$ is even. If $m$ is a power of $2$,  then  $G$ has a factor group isomorphic to $Q$ and $m = 4$ by Lemma \ref{vP}.
  Thus, $m \in \{4, 6\}$.
  
Clearly,  $Q \cap A$ is the (unique)  Sylow $2$-subgroup of $A$; let  $Z_0 = \cent{Q \cap A}P$ and $D = [Q \cap A, P]$, where $P\in \syl 3G$. 
By coprime action, $Q\cap A = Z_0 \times D$.
As $AP \nor G$, by the Frattini argument we get $G = A \norm GP$, and this implies that  both $Z_0$ and $D$ are normal subgroups of $G$.  
Let  $K$ be  the  $2$-complement of $A$.
By part (4) of Lemma~\ref{5/6}, we have  $Z_0K/K \leq \z{G/K}$,  so $[Z_0, G]\leq Z_0 \cap K = 1$  and hence $Z_0 \leq \z G$.    

If $m = 4$, then $P \leq A$ and 
$Z_0 =  Q\cap A = \z Q$ and we have case (b1).

Now, we assume that $m = 6$ and  
we  first prove that   $Q' \leq D$. In fact, setting  $L = DKP$ and observing that $L \nor G$, if $Q' \not\leq D$ then  $G/L \cong Q/D$ is  a nonabelian $2$-group, 
so by Lemma~\ref{vP} $\N{G/L}  {\color{blue} = } X/L$, where $X/L = \z{G/L}$.
Hence $A = \N G \leq X \cap A$, so $A \leq X$, which is  a contradiction 
because $4$ divides $[G:X]$. Thus, $Q' \leq D$ and, in particular, $D \neq 1$. 
Let  now $\o G = G/Z_0K$. 
Note that $\o A \subseteq \N{\o G}$ and  that $\N{\o G}$ is a subgroup of $\o G$ by
Theorem~\ref{zyd}, so $\o m = [\o G : \N{\o G}]$ divides $6$.  Since  $\o Q \cong Q$ is nonabelian, the argument in the second
paragraph of this proof  yields $\o m = 6$ and hence $\P{\o G} = \frac{5}{6}$.
Therefore, since $D$ and $\o A$ are isomorphic  $G/A$-modules,  without loss of generality we can assume  $Z_0K= 1$.
Thus, $A = D \neq 1$,  $|P| = 3$ and $\cent AP = 1$.
By part (1) of Lemma~\ref{5/6} there  exists an involution  $y \in \norm QP$,
such that $Q = A \la y \ra$ and such that $H = P\la y \ra$ is a complement of $A$ in $G$; in particular, 
$|H| = 6$. If $H \cong S_3$, then by Proposition~\ref{S3classify} we have case (b2). 
In order to conclude this part of the proof, we can hence assume $H \cong C_6$ and $\class(Q)  \geq 3$, since if $\class(Q) \leq 2$, then $[D,Q,Q] = 1$, so we have case (b3) with $D = C$ and $B= 1$.  
Hence, $G$ satisfies all the conditions of Setting~\ref{hyp1}, and Proposition~\ref{C6classify} yields case (b3). 

\medskip
We now prove the other implication.
In case (a), $G$ is an $\mathcal{A}$-group with $m = [G:\fitt G] \leq 6$.
If $m \neq 5$, then $\fitt G \subseteq \N G$ by Lemma~\ref{56connect} and Lemma~\ref{brough}.
Thus, $\P G \leq 1 - \frac{1}{m}  < \mathfrak{a}$.
If $m = 5$, then $\P G < \mathfrak{a}$ by part (3) of Lemma~\ref{Agroups}. 

So, we assume (b) and we prove  that $A \subseteq \N G$, which implies $\P G \leq \frac{5}{6} < \mathfrak{a}$. Let $K$ be the $2$-complement of $A$.
We observe that $A = D \times Z_0 \times K$, with  $Z_0 \leq \z G$ and $D\nor G$ (as $G = A \norm GP$). 
For $a \in A$, we write $a = d z k$ with $d \in D$, $z \in Z_0$ and $k \in K$, and we observe that by Lemma~\ref{56connect}
$a \in \N G$ if and only if $a_2 = dz \in \N G$.
If we are in case (b1), then $Q\cap A \subseteq \N G$ by Lemma~\ref{brough}.
For the remaining two cases, we observe that by Lemma~\ref{nonvancent} $dz \in \N G$ if and only if $d \in \N G$  and that,
setting $\o G = G/Z_0K$,  $d \in \N G$ if and only $\o d \in \N{\o G}$ by Lemma~\ref{dpnv}.
Since $\o G$ satisfies Setting~\ref{hyp}, in case (b2) we have $\o D \subseteq \N{\o G}$ by Proposition~\ref{S3classify},
hence $A \subseteq \N G$.
So, we are left with case (b3) and $G/A \cong C_6$. If $B = 1$, then $\class(\o Q) \leq  \class(Q) = 2$ and hence  $\o D \subseteq \N{\o G}$ by part (2) of Lemma~\ref{5/6}.
If $B \neq 1$, then the assumptions on $B$ and $C$ imply that $\class(Q) = 3$. 
Therefore, $\o G$ satisfies all the conditions of Setting \ref{hyp1}, so $\o D \subseteq \N{\o G}$ by Proposition~\ref{C6classify}
and hence $A \subseteq \N G$, concluding the proof. 
\end{proof}

Finally, we give a  description of the  groups of case (a) of Theorem A.

\begin{thm}\label{Asmall}
  Let $G$ be an $\mathcal{A}$-group with trivial center. 
  Then $2 \leq [G:\fitt G] \leq 6$ if and only if $G$ has an  abelian normal  subgroup $A$ such that one of the following occurs.
  \begin{description}
  \item[(1)] $G$ is a Frobenius group with kernel $A$ and complement $H \cong C_m$, $2 \leq m \leq 6$.
  \item[(2)] $G = A \rtimes Q$, with $A$ of odd order and either
    \begin{description}
    \item[(2i)]$Q = \la y \ra \cong C_4$ and $A = C \times D$, with  $C = \cent A{\la y^2 \ra}$,  $D = [A, \la y^2 \ra]$ and 
    both $C \rtimes Q/\la y^2 \ra$ and $D \rtimes Q$ are Frobenius groups; 
    or
    \item[(2ii)]  $Q = \{1, y_1, y_2, y_3 \}$, with $y_i$ involutions, and $A = A_1 \times A_2 \times A_3$, $A_i = \cent A {y_i}$,  and
    $(\prod_{j \neq i} A_j)\rtimes \la y_i\ra$ are  Frobenius groups,  for $i = 1,2, 3$.
    \end{description}
    
  \item[(3)] $G = A P_0Q_0$,  for $P_0, Q_0 \leq G$,  $|P_0| = 3$, $|Q_0| = 2$,
    $[P_0, Q_0] = 1$, $A = B \times C \times D$ with $B = \cent A{P_0}$, $C = \cent A{Q_0}$,
    and both  $BD \rtimes Q_0$ and $CD \rtimes P_0$ are Frobenius groups. 
  \item[(4)] $G = A PQ$, with $P \in \syl 3G$ and $Q \in \syl 2G$, $|Q| = 2$, $|P/P\cap A| = 3$, $[P, Q] = P$, $A = C \times D$, with $C = \cent AP$ and $D = [P, A] = \cent AQ$
    normal subgroups of $G$, and  both $CP \rtimes Q$ and $D\rtimes P/P\cap A$ are  Frobenius groups.
  \end{description}
\end{thm}
  \begin{proof}
    Let $G$ be an $\mathcal{A}$-group such that $Z = \z G = 1$. 
    If $G$ satisfies one of the conditions $(1)-(4)$, then $G$ is nonabelian,  $A \leq \fitt G$ (actually, $A = \fitt G$) and $[G:A] \leq 6$.

    \smallskip
    Conversely, writing $A = \fitt G$, we assume that  $ 2 \leq  [G:A] \leq 6$ and show that one of $(1)-(4)$ follows.
    If $[G:A] = r \in \{2,3,5\}$ and $R \in \syl rG$, then $\cent AR = 1$ and  $A \cap R= 1$,  because $Z = 1$ and $R$ is abelian.  
    So, $G$ is a Frobenius groups with kernel $A$ and cyclic complement of order $r$.

    Next, we assume $[G:A] = 4$. Then $G = AQ$ where $Q \in \syl 2G$ and, as above, $\cent AQ = 1$ and  $A \cap Q= 1$, so $|Q| = 4$.
    If $Q$ is cyclic, then either $G$ is a Frobenius group with complement $Q$ or, denoting by $Y$ the subgroup
    of order $2$ of $Q$, by coprimality $A = C \times D$, where $C = \cent AY$ and $D = [A, Y]$ are nontrivial normal subgroups of $G$,
    and both $CQ/Y$ and $DQ$ are Frobenius groups.
    If instead $Q$ is elementary abelian, say $Q = \{1, y_1, y_2, y_3 \}$, with $y_i$ involutions, it is well known that $A$
    is the  product of the normal  subgroups $A_i = \cent A{y_i}$, $i =1, 2,3$, of $G$.
    Note that,  for distinct indices $i$, $j$ and $k$,  $A_i \cap A_jA_k = \cent{A_j}{y_i}  \cap \cent{A_k}{y_i} =1$ because any
    two involutions of $Q$ generate $Q$ and $\cent AQ = 1$; hence, $A = A_1 \times A_2 \times A_3$. 
    Observe also that at most one of the subgroups $A_i$ can be
    trivial,  and that when this happens  $G$ is the direct product of two Frobenius groups with complements of order $2$.
    In any case, $(\prod_{j \neq i} A_j)\la y_i\ra$ are  Frobenius groups,  for $i = 1,2, 3$.

    So, we assume $[G:A] = 6$. Let $P \in \syl 3G$, $Q \in \syl 2G$ and observe that $\cent AP \cap \cent AQ \leq Z = 1$. 
    As $AP \nor G$, then $G = A \norm GP$ and hence   $AQ \cap \norm GP$ contains a Sylow $2$-subgroup $Q_0$ of $\norm GP$ (because $[\norm{G}{P}:AQ\cap \norm{G}{P}]=[G:AQ]=3$).
    But $[Q_0 \cap A, P] \leq \oh 2G \cap P = 1$, so $Q_0 \cap A \leq \cent AQ \cap \cent AP = 1$ and hence $Q = (Q \cap A) \rtimes Q_0$,
    with $|Q_0| = 2$.

If  $G/A \cong C_6$, by a similar argument we get   $P = (P \cap A) \rtimes P_0$,
with $|P_0| = 3$ and $P_0 \leq \norm GQ$; so $[P_0, Q_0] =1$ and $P_0Q_0$, a cyclic group of order $6$, is a complement of $A$ in $G$.  Since both $P$ and $Q$ are abelian, coprimality considerations give
$A = B \times C \times D$ with $B = \cent AP = \cent A{P_0}$,  $C = \cent AQ = \cent A{Q_0}$ and by the three subgroups lemma $D = [A, P_0, Q_0] = [A, Q_0, P_0]$ (note also that $\cent A{P_0} \cap \cent A{Q_0} =1$ implies $\cent A{Q_0} = [\cent A{Q_0}, P_0] \leq [A, P_0]$, so $\cent {[A, P_0]}{Q_0} = \cent A{Q_0}$). 
Since $Z = 1$, we have that both $BDQ_0$ and $CDP_0$ are Frobenius groups with complements, respectively,  of order $2$ and of order $3$. 

Finally, we assume  $G/A \cong S_3$. Observe that
$\cent G{A\cap Q}$ is a normal subgroup of $G$ that contains $AQ$; but $AQ/A$ is a non-normal maximal  subgroup of $G/A$, hence 
$A \cap Q \leq Z = 1$ and $Q = Q_0$ has order $2$. Let $C = \cent AP$ and $D = [P, A]$. 
Since $P$ is abelian, $A=C\times D$.
We also have $P\cap A = [P \cap A, Q]$, because $\cent {P\cap A}Q \leq Z = 1$, so $[P, Q] = P$ (as $G/A$ is nonabelian).
Hence, $\cent {CP}Q = \cent CQ \cent PQ = 1$, so  $CPQ$ is a Frobenius group with kernel $CP$  and complement $Q$. 
Since $\cent DP =  C \cap D = 1$, also the semidirect product $D \rtimes P/P\cap A$ is a Frobenius group with complement $P/P\cap A$ of order $3$.
We finally observe that $D \leq \cent AQ$,  because if $[D, Q] \neq 1$ then $PQ/P \cap A$ would act fixed point freely on $[D, Q]$, giving
$G/A \cong PQ/(P\cap A)$ cyclic, a contradiction.
So, by Dedekind's law $\cent AQ = D (\cent AQ \cap C) = D$, concluding the proof.     
  \end{proof}
  

  From part (1) of Lemma~\ref{Agroups} we have the following.
  
  \begin{cor}\label{corAgr}
    Let $G$ be a nonabelian $\mathcal{A}$-group.
    Then $[G:\fitt G] \leq 6$ if and only if $G/\z G$ satisfies one of the conditions $(1)$ to $(4)$ of Theorem~\ref{Asmall}. 
  \end{cor}

\end{document}